\documentclass{amsart}
\usepackage[all]{xy}

\theoremstyle{plain}
\newtheorem{theorem}{Theorem}
\newtheorem{lemma}[theorem]{Lemma}
\newtheorem{proposition}[theorem]{Proposition}

\theoremstyle{definition}
\newtheorem{definition}[theorem]{Definition}
\newtheorem{remark}[theorem]{Remark}

\DeclareMathOperator{\act}{\alpha}
\DeclareMathOperator{\Aut}{Aut}
\DeclareMathOperator{\Stab}{Stab}
\DeclareMathOperator{\Sym}{Sym}
\DeclareMathOperator{\dimh}{\dim_{\mathcal{H}}}

\begin{document}

\title{Finitely constrained groups of maximal Hausdorff dimension}

\author{Andrew Penland}
\address{Dept. Of Mathematics, Texas A\&M University, College Station, TX 77843-3368, USA}
\email{apenland@math.tamu.edu}

\author{Zoran \v{S}uni\'c}
\email{sunic@math.tamu.edu}
\thanks{This material is based upon work supported by the National Science
Foundation under Grant No. DMS-1105520. }

\begin{abstract}
We prove that if $G_P$ is a finitely constrained group of binary rooted tree automorphisms (a group binary tree subshift of finite type) defined by an essential pattern group $P$ of pattern size $d$, $d \geq 2$,  and if $G_P$ has maximal Hausdorff dimension (equal to $1-1/2^{d-1}$), then $G_P$ is not topologically finitely generated. We describe precisely all essential pattern groups $P$ that yield finitely constrained groups with maximal Haudorff dimension. For a given size $d$, $d \geq 2$, there are exactly $2^{d-1}$ such pattern groups and they are all maximal in the group of automorphisms of the finite rooted regular tree of depth $d$.
\end{abstract}


\maketitle


\section{Introduction and main results}

Finitely constrained groups were introduced by Grigorchuk in 2005~\cite{grigorchuk:unsolved}. They are compact groups of rooted tree automorphisms that may be defined by finitely many forbidden patterns on the tree, i.e., they are groups of rooted tree automorphisms that are also tree subshifts of finite type (this is why they are sometimes called groups of finite type, the term that Grigorchuk used originally).

Our goal is to prove the following results on finitely constrained groups of binary tree automorphisms.

\begin{theorem}\label{t:main-equiv}
Let $G_P$ be a finitely constrained group of binary rooted tree automorphisms (a group binary tree subshift of finite type) defined by an essential pattern group $P$ of pattern size $d$, $d \geq 2$. The following conditions are equivalent.
\begin{enumerate}
\item $G_P$ has maximal Hausdorff dimension (equal to $1-1/2^{d-1}$).

\item $P$ is a proper subgroup of the group $G(d)$ of automorphisms of the binary rooted tree of depth $d$ that contains the commutator subgroup of $G(d)$.

\item $P$ is a maximal subgroup of $G(d)$ that does not contain the generator $a_{d-1}$.
\end{enumerate}
\end{theorem}

\begin{theorem}\label{t:main}
Let $G_P$ be a finitely constrained group of binary rooted tree automorphisms (a group binary tree subshift of finite type) defined by an essential pattern group $P$ of pattern size $d$, $d \geq 2$. If $G_P$ has maximal Hausdorff dimension (equal to $1-1/2^{d-1}$), then $G_P$ is not topologically finitely generated.
\end{theorem}

Note that Theorem~\ref{t:main} is already known for sizes $d=2$, $d=3$, and $d=4$. The case $d=2$ is subsumed in the earlier results of the second author~\cite{sunic:pattern}, and the cases $d=3$ and $d=4$ are subsumed in the results of Bondarenko and Samoilovych~\cite{bondarenko-s:size3}.

It follows from~\cite[Proposition~2.7]{bartholdi:branch-rings} (and independently from~\cite[Proposition~6]{sunic:hausdorff}) that the possible values of the Hausdorff dimension of finitely constrained groups of binary tree automorphisms defined by pattern groups of pattern size $d$ are limited to the set
\begin{equation}\label{e:possible}
\{1,~1-\frac{1}{2^{d-1}},~1-\frac{2}{2^{d-1}},~\dots~,~ \frac{1}{2^{d-1}},~0\}
\end{equation}
and that the value 0 is attained only for finite groups. The value 1 is attained only for the entire group $\Aut(X^*)$ of automorphisms of the binary tree (this group is finitely constrained as it is defined by allowing all patterns). Thus, the situation is clear for the two extreme values of the Hausdorff dimension, 0 and 1, and Theorem~\ref{t:main-equiv} and Theorem~\ref{t:main} address the groups of maximal Hausdorff dimension different than 1, dimension $1-1/2^{d-1}$.

The bound in Theorem ~\ref{t:main} is the best possible, as the first author has recently discovered an example of a topologically finitely generated, finitely constrained group defined by patterns of size $d = 5$ and having Hausdorff dimension $14/16 = 1 - 2/2^{d-1}$. Previously, Grigorchuk proved that the topological closure of the first Grigorchuk group has Hausdorff dimension $5/8=1-3/2^{4-1}$~\cite{grigorchuk:jibg} and it is a finitely constrained group of binary tree automorphisms defined by a pattern group $P$ of pattern size $4$~\cite{grigorchuk:unsolved}. Moreover, for every $d \geq 4$, the second author constructed finitely constrained groups of binary tree automorphisms defined by pattern groups of pattern size $d$ that are topologically finitely generated and have Hausdorff dimension equal to $1-3/2^{d-1}$~\cite{sunic:hausdorff}. 

While a finitely constrained group, other than the group of all tree automorphisms $\Aut(X^*)$, cannot have Hausdorff dimension 1, Ab\'{e}rt and Vir\'{a}g~\cite{abert-v:dimension} showed that, with probability 1, three random binary tree automorphisms generate a subgroup whose closure has Hausdorff dimension 1. Siegenthaler~\cite{siegenthaler:irrational} constructed the first explicit examples of topologically finitely generated groups with Hausdorff dimension equal to 1. The construction of Siegenthaler is based on combining together, in a single group, a sequence of spinal groups~\cite{bartholdi-s:wpg} whose closures have higher and higher Hausdorff dimension (arbitrarily close to 1).

The paper has the following outline. Section~\ref{s:aut} and Section~\ref{s:hausdorff} provide the necessary background on groups of binary rooted tree automorphisms, finitely constrained groups and their Hausdorff dimension. We prove Theorem~\ref{t:main-equiv} in Section~\ref{s:maximal}. We then provide a property of the commutator $[P,P]$ for maximal subgroups $P$ of the group $G(d)$ of automorphisms of the binary rooted tree of depth $d$ (Section~\ref{s:commutator}, Proposition~\ref{p:no-adad}), which is then used, along with Theorem~\ref{t:main-equiv} and the condition of Bondarenko and Samoilovych (Theorem~\ref{t:bs-condition}), to prove Theorem~\ref{t:main} in Subsection~\ref{s:main-proof}.


\section{The groups $G=\Aut(X^*)$ and $G(d)=\Aut(X^{[d]})$}\label{s:aut}

We emphasize that all our considerations and claims are limited to the binary rooted tree case. even though most of the notions make sense and many (but not all!) results that we use or prove are valid on trees of higher arity, especially if one limits the considerations to the case of $p$-adic automorphisms of the $p$-ary rooted tree, for some prime $p$.

\subsection{The group $G=\Aut(X^*)$ of binary rooted tree automorphisms}

Let $X=\{0,1\}$. For $n \geq 0$, we denote by $X^n$ the set of words of length $n$ over $X$ (with the empty word $\emptyset$ having length 0), and we write $|w| = n$ if $w \in X^n$. Let $X^*$ be the set of all words over $X$. The set $X^*$ naturally has the structure of a binary rooted tree with the elements of $X^*$ as the vertices, edges given by $\{(w,wx)\}_{w \in X^*, x \in X}$ and $\emptyset$ as the root. Each vertex $w \in X^*$ has two children, $w0$ and $w1$.

The group $\Aut(X^*)$ consists of all automorphisms of the graph $X^*$ (such automorphisms necessarily preserve the root, the length of words, levels of the tree and the prefix relation on words). We denote $\Aut(X^*)$ by $G$ and write the action of $G$ on the vertices of $X^*$ as a left action. We identify the symmetric group $\Sym(X) = ((),(01))$ with the cyclic group $C_2 = \{0,1\}$.

The self-similarity of the tree $X^*$ leads naturally to the notion of self-similarity for subgroups of $G$. For any vertex $w \in X^*$, the subtree $wX^*$ can be viewed as a copy of $X^*$ rooted at $w$, so the groups $\Aut(wX^*)$ and $\Aut(X^*)$ are isomorphic. For $w \in X^*$ and $g \in \Aut(X^*)$, we define the \emph{section of g at w} to be the unique element $g_w \in \Aut(X^*)$ such that $g(wv)=g(w)g_w(v)$, for all $v \in X^*$. In other words, the action of $g_w$ on $X^*$ corresponds to the ``tail of the action'' of $g$ on $X^*$ behind the prefix $w$. For any word $w= x_1x_2\cdots x_n \in X^*$, the action of $g \in G$ on $w$ is expressed via its sections as \[ g(x_1\cdots x_n) = g_{\emptyset}(x_1)g_{x_1}(x_2)g_{x_1x_2}(x_3) \cdots g_{x_1x_2\cdots x_{n-1}}(x_n). \] The chain rule formula $(hg)_{u} = h_{g(u)}g_{u}$ and the inversion formula $(g^{-1})_{u} = (g_{g^{-1}(u)})^{-1}$ hold for any $h,g \in \Aut(X^*)$ and any $u \in X^*$.

Of particular interest are subgroups of $\Aut(X^*)$ which contain all sections of all their elements.

\begin{definition}
A subgroup $H$ of $\Aut(X^*)$ is called \emph{self-similar} if whenever $h \in H$ and $w \in X^*$, $h_w \in H$.
\end{definition}

For a vertex $v \in X^*$, the \emph{stabilizer of v} is defined as $\Stab_G(v) = \{g \in G \mid g(v) = v \}$. The \emph{level $n$ stabilizer} is denoted by $G_n$ and is equal to
\[
 G_n = \Stab_G(X^n) = \{\ g \in G \mid g(v) = v, \text{ for }v \in X^n\  \} =  \bigcap_{v \in X^n} \Stab_G(v).
 \]
Note that, for every $w \in X^*$, $\Stab_G(w)$ is isomorphic to $G$ via the map $h \rightarrow h_w$. If $H$ is a subgroup of $G$, the level $n$ stabilizer of $H$ is $H_n = \Stab_H(X^n) = G_n \cap H$.

The action of $g \in G$ on $X^*$ restricts to an action of $\Sym(X)$ on $X$, yielding a homomorphism $\alpha: G \rightarrow C_2$. For $g \in G$, we call $\alpha(g)$ the \emph{root action of g}, and define the \emph{activity of g at v} to be $\alpha_v(g) = \alpha(g_v)$. The activity of $g$ at $v$ is also called the \emph{label of g at v} and is sometimes denoted $g_{(v)}$.

\begin{definition}
For a finite set $J \subseteq \{0,1,\dots\}$ and an element $g \in G$, define the activity of $g$ within $J$ to be the sum, modulo 2, of the activities of $g$ on all vertices on levels from $J$, i.e.,
\[
 \alpha_J(g) = \sum_{j \in J} \sum_{v \in X^j} {g_{(v)}}.
\]
\end{definition}

Note that the map $\alpha_J: G \to C_2$ is a homomorphism.

\subsection{The group $G(d)=\Aut(X^{[d]})$ of finite binary rooted tree automorphisms} Let
\[
 X^{[d]} = \bigcup_{i=0}^d X^i \qquad \text{and} \qquad X^{(d)} = \bigcup_{i = 0}^{d-1} X^i .
\]
The set $X^{[d]}$ corresponds to the finite subtree of $X^*$ with $(d+1)$ levels rooted at $\emptyset$ such that every vertex (except for the leaves) has $|X|$ children. The group of automorphisms of the tree $X^{[d]}$ is denoted by $G(d)$. Note that $G(d)$ is isomorphic to the quotient $G/G_d$. The following properties of $G(d)$ are well known.

\begin{proposition}
For $i=0,\dots,d-1$, let $a_i$ be the automorphism in $G(d)$ such that  $(a_i)_{(w)}$ is nontrivial if and only if $w = 0^i$.
\begin{enumerate}
\item The finite group $G(d)$ is generated by the set $\{a_i\}_{i=0}^{d-1}$.

\item The group $G(d) = \underbrace{C_2 \wr C_2 \wr \dots \wr C_2}_d$ has a presentation
\[
 G(d) = \langle a_0,a_1,\dots,a_{d-1} \mid a_i^2,~\text{ for } 0\leq i \leq d-1,~  [a_j^{a_i},a_k],~\text{ for } 0 \leq i<j,k \leq d-1 \rangle.
\]

\item $G(d)/[G(d),G(d)]$ is an elementary abelian 2-group of rank $d$.
\end{enumerate}
\end{proposition}

Most of the notions defined for $G$ naturally transfer to the group $G(d)$ of automorphisms of the finite tree $X^{[d]}$. For a word $w$ of length smaller than $d$ and $g \in G(d)$, we define the section $g_w$ to be the unique element of $G(d - |w|)$ such that $g(wv) = g(w)g_{w}(v)$, for all $v \in X^{([d-|w|])}$. The vertex labels and the activity of elements of $G(d)$ are defined exactly as they are for elements of $G$ (limited up to and including level $d-1$). The chain rule and the inversion formula are valid in this context too. For $n=0,\dots,d$, the stabilizer $\Stab_{G(d)}(X^n)$ of level $n$ is denoted by $G_n(d)$, and for any subgroup $P$ of $G(d)$, the stabilizer $\Stab_{P}(X^n)=G_n(d) \cap P$ is denoted by $P_n$.


\section{Hausdorff dimension and finitely constrained groups}\label{s:hausdorff}

\subsection{Metric on $G$ and Hausdorff dimension of closed subgroups of $G$}

Hausdorff dimension is a well-known concept from fractal geometry which can be defined for any metric space. In this section we will consider Hausdorff dimension only as it applies to self-similar groups of binary rooted tree automorphisms. Abercrombie~\cite{abercrombie:subgroups} was the first to consider the Hausdorff dimension of closed subgroups of a profinite group with respect to the natural profinite metric structure. Barnea and Shalev considered Hausdorff dimension in pro-$p$ groups and provided a formula which gives the Hausdorff dimension of a closed subgroup via a sequence of finite quotients~\cite[Theorem 2.4]{barnea-s:hausdorff} as
\begin{equation}\label{e:liminf-p}
\dim_{\mathcal{H}}(H) = \liminf_{n \rightarrow \infty} \frac{\log_2[H:H_n]}{\log_2[G:G_n]}.
\end{equation}

In general, different metrics lead to different Hausdorff dimension functions for a space. Thus, we need to be careful and spell out precisely the metric on $G=\Aut(X^*)$ for which~\eqref{e:liminf-p} correctly expresses the Hausdorff dimension. Using the notion of activity from the previous section, we can define the \emph{portrait map} $\phi: G \to (C_2)^{X^*}$ given by $\phi(g) = (\alpha_{v}(g))_{v \in X^*}$. The portrait map is bijective and therefore identifies $G$ with the compact space $(C_2)^{X^*}$. Via this identification, $G$ is equipped with a metric $d$ given by $d(g,h) = 0$ if $g = h$ and
 \[
d(g,h) = \frac{1}{[G:G_n]} \text { if } g \neq h \in G,
\]  where
\[ n = \inf \{\ k \mid \text{ there is } u \in X^k \text{ with } g_{(u)} \neq h_{(u)} \}.
 \]
Informally, two elements of $G$ are close in this metric if their actions are identical on a large subtree rooted at $\emptyset$. Since in the binary case $[G:G_n]=2^{2^n-1}$, we may rewrite the formula for Hausdorff dimension in the form
\begin{equation}\label{e:liminf}
\dim_{\mathcal{H}}(H) = \liminf_{n \rightarrow \infty} \frac{\log_2 [H:H_n]}{2^n-1},
\end{equation}
for any closed subgroup $H$ of $G$.

A closed subgroup of $G$ is called \emph{topologically finitely generated} if it is the closure of a finitely generated subgroup of $G$.

Note that, since $G$ is compact Hausdorff group, so is each of its closed subgroups.

\subsection{Finitely constrained groups}

Finitely constrained groups, introduced by Grigorchuk~\cite{grigorchuk:unsolved}, combine the group theoretic, topological, and symbolic dynamics aspects of tree automorphisms. Namely, they are subgroups of $G$, they are topologically closed (with respect to the metric on $G$ defined in the previous section), and they are closed under self-similarity. The last two properties make them tree subshifts (see~\cite{aubrun-b:finite,ceccherini-c-f-s:cellular-long}). Note that in general, a tree subshift has no group structure. It is well-known that tree subshifts can always be constructed by specifying a set of forbidden patterns, and we now make precise all these notions, but only in the limited setting of the binary rooted tree $X^*$, with labels on the vertices coming from the alphabet $\Sym(X)=C_2$.

For $d > 0$, a \emph{pattern of size $d$} is a map from $X^{(d)}$ to $C_2$. For $g \in G$, we say a pattern $p$ of size $d$ \emph{appears at w in g} if $g_{(wv)} = p(v)$, for $v \in X^{(d)}$. Given a subset $S \subseteq G$, we say a pattern $p$ \emph{appears in S} if $p$ appears at some vertex in some element of $S$. Given a set $\mathcal{F}$ of patterns, we can define a \emph{tree subshift} $Y_{\mathcal{F}}$ to be the subset of $G$ such that no pattern in $\mathcal{F}$ appears in any $y \in Y_{\mathcal{F}}$. Moreover, every tree subshift $Y$ has a defining set $\mathcal{F}$ of forbidden patterns. If the set $\mathcal{F}$ can be taken to be finite, then $Y_{\mathcal{F}}$ is a \emph{tree subshift of finite type}.

By taking possible extensions of patterns as needed, we can assume that all forbidden patterns for a tree subshift of finite type are of the same size. The complement of the finite set $\mathcal{F}$ in $(C_2)^{X^{(d)}}$ is the \emph{set of allowed patterns}.

\begin{definition}
A \emph{finitely constrained group} is a subgroup of $G=\Aut(X^*)$ which is a tree subshift of finite type.
\end{definition}

The allowed patterns of a finitely constrained group $H$ form a subgroup of $G(d)$ isomorphic to $H/H_d$.

Conversely, any subgroup of $G(d)$, for $d \geq 1$, corresponds to a set of patterns of size $d$, which may be used as the set of allowed patterns to construct a finitely constrained group. Indeed, in the finite context, the group $G(d)$, $d \geq 1$, corresponds bijectively to the set of functions $(C_2)^{X^{(d)}}$ under the (finite version of the) portrait map given by $\phi(g) = (\alpha_{v}(g))_{v \in X^{(d)}}$. We want to consider only pattern groups in which all patterns are actually used in the tree subshift that they define.

\begin{definition}
A \emph{pattern group of size d} is a subgroup of $G(d)$, $d \geq 1$. A pattern group $P$ is an \emph{essential pattern group} if for all $g \in P$ and $i = 0,1$, there exists $h_i \in P$ such that $(h_i)_{(w)} = g_{(iw)}$ for all $w \in X^{(d-1)}$.
\end{definition}

Given a pattern group $P$, we define the \emph{self-similar group defined by P} as the group whose allowed patterns of size $d$ are precisely the elements of $P$. Note that any pattern group can be reduced to an essential pattern group which defines the same self-similar group. Thus every finitely constrained group is defined by some essential pattern group.

Bondarenko and Samoilovych~\cite{bondarenko-s:size3} provide algorithms to determine finiteness and level-transitivity of groups defined by essential pattern groups. While they do not state explicitly the following simple formula for the Hausdorff dimension of a finitely constrained group, it may be easily inferred from parts of the proof of their criterion for finiteness of $G_P$ (see~\cite[Proposition~1]{bondarenko-s:size3}) and Equation~\eqref{e:liminf}.

\begin{lemma}\label{l:bs-formula}
Let $P$ be an essential pattern group with patterns of size $d$, and let $G_P$ be the finitely constrained group defined by $P$. Then
 \[
 \dim_{\mathcal{H}} G_{P} = \frac{\log_{2} |P_{d-1}|}{2^{d-1}}.
 \]
\end{lemma}

\begin{remark}
For each self-similar group $H$ of binary tree automorphisms, the map $\psi: H_1 \to H \times H$ given by $h \mapsto (h_0,h_1)$ is an embedding. It is common to identify $H_1$ with its image in $H \times H$ under $\psi$. The formula from~\cite{sunic:hausdorff} for Hausdorff dimension of a finitely constrained group $H$ of binary rooted tree automorphism defined by forbidden patterns of size $d$ states that
\[
 \dimh(H) = \frac{r-t+1}{2^{d-1}},
\]
where $2^t=[H \times H:H_1]$ and $2^r=[H:H_{d-1}]$. Combining this formula with the formula in Lemma~\ref{l:bs-formula}, we obtain a new relation
\[
 2 \cdot [H:H_{d-1}]= |P_{d-1}|\cdot [H\times H : H_1],
\]
where $P$ is the essential pattern group of size $d$, defining $H$. Since $[H:H_{d-1}] = [P:P_{d-1}] = |P|/|P_{d-1}|$, we also have
\begin{equation}\label{e:new-relation}
 2 |P| = |P_{d-1}|^2 \cdot [H\times H : H_1],
\end{equation}
\end{remark}


\section{Maximal Hausdorff dimension corresponds to maximal subgroups}\label{s:maximal}

The maximal subgroups of $G(d)$ correspond bijectively to nonempty subsets of $\{0,\dots,d-1\}$ as follows. For a nonempty subset $J \subseteq \{0,\dots,d-1\}$, define a subgroup $P_J$ by
\[
 P_J = \{\ g \in G(d) \mid \act_J(g) = 0 \ \}.
\]
In other words, $P_J$ is the kernel of the nontrivial homomorphism $\alpha_J: G(d) \to C_2$. Conversely, every maximal subgroup $P$ is the kernel of a nontrivial homomorphism $\theta: G(d) \to C_2$. If we define $J=\{ j \in \{0,\dots,d-1\} \mid a_j \not \in \ker \theta \}$, then $P=P_J$. Note that there are $2^d-1$ maximal subgroups of $G(d)$, and Theorem~\ref{t:main-equiv} claims that only $2^{d-1}$ of them, those that do not contain $a_{d-1}$, can be used as essential pattern groups. Moreover, no other group above the commutator, with the exception of the whole group $G(d)$, can be used as an essential pattern group. The condition that $a_{d-1}$ is not in $P_J$ is equivalent to the condition that $d-1$ is in $J$.

In order to prove Theorem~\ref{t:main-equiv}, we use the following technical, but straightforward result (an exercise in using the chain rule for permutational wreath products).

\begin{lemma}\label{l:conjugating}
For $g \in G(d)$ and $h \in G_{d-1}(d)$, the conjugate $h^g$ is in $G_{d-1}(d)$ and, for $v \in X^{d-1}$,
\[
 \left(h^g\right)_{(v)} = h_{(g(v))}.
\]
\end{lemma}

\begin{proof}
We have
\begin{align*}
 \left(h^g\right)_{v} &= \left(g^{-1}hg\right)_v = \left(g^{-1}\right)_{h(g(v))}h_{g(v)}g_v =
 \left(g^{-1}\right)_{g(v)}h_{g(v)}g_v = \\
 &= \left(g_{g^{-1}g(v)}\right)^{-1}h_{g(v)}g_v = \left(g_{v}\right)^{-1}h_{g(v)}g_v,
\end{align*}
which implies that
\[
 \left(h^g\right)_{(v)} = \left(g_{(v)}\right)^{-1}h_{(g(v))}g_{(v)} = h_{(g(v))}. \qedhere
\]
\end{proof}

The following result is also of use.

\begin{proposition}\label{p:transitivity}
Let $H$ be a finitely constrained group on the binary rooted tree $X^*$. The following are equivalent.
\begin{enumerate}
\item $H$ is infinite.

\item $H$ acts transitively on all levels of the tree $X^*$.

\item The Hausdorff dimension of $H$ is positive.
\end{enumerate}
\end{proposition}

The equivalence of (i) and (ii) holds for arbitrary self-similar subgroups of $G$~\cite[Lemma~3]{bondarenko-al:classification32}) (it is important for this equivalence that the tree is binary). The equivalence of (i) and (iii) follows from~\cite[Theorem 4(a)]{sunic:hausdorff}.

\begin{proof}[Proof of Theorem~\ref{t:main-equiv}]
\textbf{(i) implies (iii)}. By Lemma~\ref{l:bs-formula},
\[
 \frac{\log_2 |P_{d-1}|}{2^{d-1}} = 1 - \frac{1}{2^{d-1}},
\]
which gives
\[
 |P_{d-1}| = 2^{2^{d-1}-1}.
\]
Since
\[
 G_{d-1}(d) \cong \prod_{v \in X^{d-1}} C_2
\]
is the elementary abelian group of rank $2^{d-1}$, we see that $[G_{d-1}(d):P_{d-1}]=2$ and $P_{d-1}$ is maximal in $G_{d-1}(d)$.

Every maximal subgroup of $G_{d-1}(d)$ has the form
\[
 M_V = \{ g \in G_{d-1}(d) \mid \beta_V(g)=0 \},
\]
where $V \subseteq X^{d-1}$ is a nonempty set of vertices on level $d-1$ and \[
 \beta_V(g) = \sum_{v \in V} g_{(v)}
\]
is the total activity, mod 2, of $g$ at the vertices in $V$. The set of vertices $V$ uniquely determines the group $M_V$ (different sets define different maximal subgroups).

We claim that $P_{d-1} = M_{X^{d-1}}$ (to say it differently, we claim that $P_{d-1}$ consists of those elements $g$ in $G_{d-1}(d)$ for which $\alpha_{\{d-1\}}(g)=0$).

Let $P_{d-1}=M_V$, for some nonempty subset $V \subseteq X^{d-1}$.

By Lemma~\ref{l:conjugating}, for $g \in G(d)$, we have $(M_V)^g = M_{g^{-1}V}$. Indeed,
\begin{align*}
 (M_V)^g &= \left\{ h^g \in G_{d-1}(d) \mid \sum_{v \in V} h_{(v)}=0 \right\}  \\
 &= \left\{ f \in G_{d-1}(d) \mid \sum_{v \in V} \left(f^{g^{-1}}\right)_{(v)}=0 \right\}  \\
 &= \left\{ f \in G_{d-1}(d) \mid \sum_{v \in V}\left(f\right)_{(g^{-1}(v))} =0 \right\}  \\
 &= M_{g^{-1}V}.
\end{align*}
Since $P_{d-1}$ is normal in $P$, we have, for $g \in P$,
\[
 M_V = P_{d-1} = \left(P_{d-1}\right)^g = \left(M_V\right)^g = M_{g^{-1}V}.
\]
Therefore, the set of vertices $V$ is invariant under the action of every element $g \in P$.

Since $d \geq 2$, the Hausdorff dimension $1-1/2^{d-1}$ is positive. By Proposition~\ref{p:transitivity}, this implies that $G_P$ acts transitively on all levels of the tree, which means that $P$ acts transitively on $X^{d-1}$. Since $V$ is a nonempty set of vertices that is invariant under the action of $P$, it follows that $V=X^{d-1}$, as claimed.

Since
\[
 P_{d-1} = M_{X^{d-1}} = \{ g \in G_{d-1}(d) \mid \alpha_{\{d-1\}}(g)=0 \}
\]
and $\alpha_{\{d-1\}}(a_{d-1}) = 1$, the element $a_{d-1}$ is not in $P$.

In order to show that $P$ is maximal in $G(d)$, we will show that the index $[P:P_{d-1}]$ is equal to $[G(d):G_{d-1}(d)]$, which is immediate from the following claim. For every pattern $h$ of size $d-1$ (an element $h \in G(d-1)$), there exists a pattern $g$ of size $d$ in $P$ (an element $g \in P \leq G(d)$ such that $h$ and $g$ agree on the first $d-1$ levels, i.e. levels 0 through $d-2$). It remains to prove the last claim.

For $i=0\,\dots,d-1$, let $P_i=\Stab_P(X^i)$ be the stabilizer of level $i$ in $P$. We claim that, for $i=0,\dots,d-2$, the stabilizer $P_i$ contains an element with every possible pattern of labels (vertex permutations) on level $i$. Indeed, there are elements in $P_{d-1}$ with every possible pattern on the vertices of the form $0v$ on level $d-1$ (the $2^{d-2}$ vertices in the left half of level $d-1$ in the tree), because a tree automorphism $g$ that stabilizes level $d-1$ and for which $g_{(0v)}=g_{(1v)}$, for $v \in X^{d-2}$, is necessarily an element of $P_{d-1}$. This follows from the fact that for such an element, the labels in the left half of level $d-1$ are repeated in the right half of level $d-1$, so the total activity on that level is 0. Since there are elements in $P_{d-1}$ with every possible pattern on the vertices of the form $0v$ on level $d-1$, and since every pattern in $P$ is extendable, it follows that, for $i=0,\dots,d-2$, the stabilizer $P_i$ contains an element with every possible pattern of labels on level $i$. Note that when we extend a pattern $g$ of size $d$ in $P_{d-1}$ to an allowed pattern $g'$ of size $2d-1-i$, and then restrict $g'$ to the subpattern of $g'$ of size $d$ that appears at vertex $0^{d-1-i}$, we obtain a pattern $g''$ of size $d$ in $P_i$ with labels on the vertices at level $i$ equal to the labels in the pattern $g$ on the $2^i$ vertices of the form $0^{d-1-i}v$, for $v \in X^{i}$

 (see Figure~\ref{f:extend-restrict}).
\begin{figure}[!ht]
\[
 \xymatrix@R=7pt@C=20pt{
 \text{\small{level} } \scriptstyle{0} &
 &
 &
 &
 *{} \ar@{-}[ddddrr]& & &
 \ar@{<..>}[dddd]_{\text{size } d}
 &
 &
  \ar@{<..>}[dddddd]^{\text{size } 2d-1-i}
 \\
 &&& &&& &&&
 \\
 \text{\small{level} } \scriptstyle{d-1-i} &
 &
 &
 *{} \ar@{-}[ddddrr] \ar@{}[l]|{0^{d-1-i}} & & & &
 &
 \ar@{<..>}[dddd]^{\text{size } d}
 &
 \\
 &&& & \ar@{}[u]|{g} && &&&
 \\
 \text{\small{level} } \scriptstyle{d-1} &
 &
 *{} \ar@{-}[rrrr]& & & & *{} &
 &
 &
 \\
 &&& \ar@{}[u]|<{g''} &&& &&&
 \\
 \text{\small{level} } \scriptstyle{2d-2-i} &
 *{} \ar@{-}[uuuuuurrr] \ar@{-}[rrrr] & & & & *{} & &
 &
 &
 }
\]
\caption{Extending and then restricting a pattern $g \in P_{d-1}$ to obtain a pattern $g'' \in P_i$}
\label{f:extend-restrict}
\end{figure}
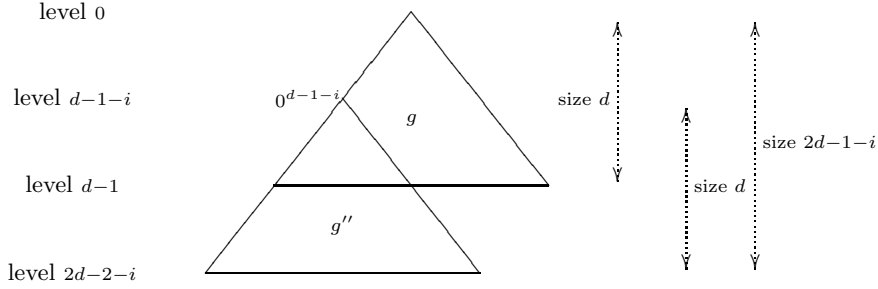
Finally, since, for $i=0,\dots,d-2$, the stabilizer $P_i$ contains elements with every possible pattern of labels on level $i$ we can obtain, by appropriate multiplication of elements, one from $P_i$, for each $i=0,\dots,d-2$, an element $g \in P$ that agrees with $h$ on the first $d-1$ levels (0 through $d-2$).

\textbf{(iii) implies (ii)} Clear, since each maximal subgroup of $G(d)$ contains the commutator subgroup of $G(d)$ (each maximal subgroup $P$ has index 2 and the quotient $G(d)/P$ is abelian).

\textbf{(ii) implies (i)} For $g \in G(d)$, we have $(a_{d-1})^g a_{d-1} = [g,a_{d-1}] \in [G(d),G(d)] \leq P$. Further, $(a_{d-1})^g a_{d-1} \in P_{d-1}$, since it stabilizes level $d-1$, and its only nontrivial activity on level $d-1$ occurs, by Lemma~\ref{l:conjugating}, at the vertices $0^{d-1}$ and $g(0^{d-1})$. Since $G(d)$ acts transitively on $X^{d-1}$, we have $M_{X^{d-1}} \leq P_{d-1}$. Therefore, either $M_{X^{d-1}} = P_{d-1}$ or $G_{d-1}(d)=P$. By Lemma~\ref{l:bs-formula}, the Hausdorff dimension of $G_P$ is $1-1/2^{d-1}$ in the former case and 1 in the latter. However, Hausdorff dimension 1 would imply that $P=G(d)$, and this contradicts the assumption that $P$ is a proper subgroup of $G(d)$. Therefore $\dimh(G_P)=1-1/2^{d-1}$.
\end{proof}

\begin{remark}
For each finitely constrained group $G_P$ defined by a maximal subgroup $P$ of $G(d)$ that does not contain $a_{d-1}$ the first level stabilizer $(G_P)_1$ is a maximal subgroup of $G_P \times G_P$. Indeed, the fact that, in this case,
\[
 |P| = \frac{|G(d)|}{2} = 2^{2^{d}-2} = \left(2^ {2^{d-1}-1}\right)^2 = |P_{d-1}|^2
\]
and~\eqref{e:new-relation} imply that
\[
 [G_P \times G_P : (G_P)_1] = 2.
\]
\end{remark}


\section{Groups of maximal Hausdorff dimension are not topologically finitely generated}\label{s:commutator}

\subsection{Commutators in maximal subgroups}

The purpose of this subsection is to prove the following result, which is a key ingredient in the proof of Theorem~\ref{t:main}.

\begin{proposition}\label{p:no-adad}
Let $P$ be a maximal subgroup of $G(d)$ that does not contain $a_{d-1}$. Then the commutator subgroup $[P,P]$ does not contain $[a_0,a_{d-1}]$.
\end{proposition}

For the duration of the rest of the section we fix a size $d$, $d \geq 2$, and a maximal subgroup $P_J$ of $G(d)$ that may serve as a pattern group of pattern size $d$. In other words, we fix $J$ that contains $d-1$. Let $J'=J -\{0\}$ (note that $J'$ is nonempty, as it contains $d-1$).

 In this section, all arithmetic operations are modulo 2.

\begin{definition}
For $i = 0,1$ and $g \in G$, let $N_i(g)$ be the total activity of the group element $g$ in the $i$th part of the tree on the levels in $J'$, i.e.,
\[
 N_i(g) = \sum_{j \in J'} \sum_{\substack{v \in X^j \\ v = iv'}} g_{v}.
\]
\end{definition}

\begin{remark}
Note that $N_0(g)$ and $N_1(g)$ are the parities of the number of nontrivial labels in the portrait of $g$ on the vertices in the left and in the right half of the tree, respectively, on the levels in $J'$. Multiplication of $g$ by $a_0$ on the right has the effect of exchanging these two parities.

We will make no use of the following observation, but it is worth noting that there is another way to think about $N_0(g)$ and $N_1(g)$. Let $g$ be expressed as a word $U$ in the generators $a_0,\dots,a_{d-1}$. A letter $a_j$ such that $j \in J'$ is called a $J'$-letter. A specific occurrence of a $J'$-letter $a_j$ in the word $U$ is declared even or odd depending on whether $\alpha(U')=0$ or $\alpha(U')=1$, respectively, where $U'$ is the suffix of the word $U$ following the given occurrence of the letter $a_j$. The number $N_0(g)$ is then the parity of the number of even occurrences and $N_1(g)$ is the parity of the number of odd occurrences of the $J'$-letters in $U$.
\end{remark}

\begin{lemma}\label{l:ni}
For $g,h\in G$ and $i=0,1$,

\begin{enumerate}
\item
\[
 N_i(gh) = N_i(h) + N_{i+\act_0(h)}(g),
\]
\item
\[
 N_i(g^{-1}) = N_{i+\act_0(g)}(g).
\]
\item
\[
 N_i([g,h]) = N_i(g) + N_{i+\act_0(h)}(g) + N_i(h) + N_{i+\act_0(g)}(h).
\]
\end{enumerate}
\end{lemma}

\begin{proof}
(i) Note that $gh = \left(a_0^{\alpha(h)} g^{a_0^{\alpha(h)}}\right) \left(a_0^{-\alpha(h)}h\right)$. Since the latter factor stabilizes level 1 (i.e., does not exchange the left and the right half of the tree) and since $a_0$ does not contribute to the activity on the levels in $J'$,
\[
 N_i(gh) = N_i(g^{a_0^{\alpha(h)}}) + N_i(h).
\]
Since the conjugate $g^{a_0^{\alpha(h)}}$ is equal to $g$, when $\alpha(h)$ is 0, and has the same labels as $g$ but exchanged between the left and the right subtree, when $\alpha(h)=1$, we have $N_i(g^{a_0^{\alpha(h)}}) = N_{i+\alpha(h)}(g)$, and the claim follows.

(ii) Follows directly from (i) by setting $h=g^{-1}$ and observing that $\alpha(1)=0$ and $\alpha(g^{-1}) = \alpha(g)$.

(iii) By using (i) and (ii)
\begin{align*}
 N_i([g,h]) &= N_i(g^{-1}h^{-1}gh) = \\
  &= N_i(h) + N_{i+\alpha(h)}(g) + N_{i+\alpha(gh)}(h^{-1}) + N_{i+\alpha(h^{-1}gh)}(g^{-1}) = \\
  &= N_i(h) + N_{i+\alpha(h)}(g) + N_{i+\alpha(g)+\alpha(h)}(h^{-1}) + N_{i+\alpha(g)}(g^{-1}) = \\
  &= N_i(h) + N_{i+\alpha(h)}(g) + N_{i+\alpha(g)+\alpha(h)+\alpha(h)}(h) + N_{i+\alpha(g)+\alpha(g)}(g) = \\
 &=
  N_i(g) + N_{i+\act_0(h)}(g) + N_i(h) + N_{i+\act_0(g)}(h). \qedhere
\end{align*}
\end{proof}

\begin{lemma}\label{l:eveneven}
For $g,h \in P_J$,
\[
 N_0([g,h]) = N_1([g,h]) = 0.
\]
More generally, for every element $f \in [P_J,P_J]$,
\[
 N_0(f)=N_1(f)=0.
\]
\end{lemma}

\begin{proof}
Let $i\in\{0,1\}$. Since $N_i(h) + N_{i+\act_0(g)}(h) = 0$, when $\act_0(g) = 0$, and $N_i(h) + N_{i+\act_0(g)}(h) = \act_{J'}(h)$, when $\act_0(g) = 1$, we have
\[
 N_i(h) + N_{i+\act_0(g)}(h) = \act_0(g) \act_{J'}(h).
\]
On the other hand, for $h \in P_J$,
\[
 \act_{J'}(h) + I_0 \act_0(h) = 0,
\]
where $I_0$ is the indicator of $0$ being in $J$ (equal to 0 or 1, when $0 \not\in J$ or $0\in J$, respectively). Therefore
\[
 N_i(h) + N_{i+\act_0(g)}(h) = I_0 \act_0(h)\act_0(g),
\]
and by symmetry
\[
 N_i([g,h]) = I_0 \act_0(g)\act_0(h) + I_0 \act_0(h)\act_0(g) = 0.
\]

Since the elements of the commutator subgroup $[P_J,P_J]$ are products of commutators and $N_0([g,h])=N_1([g,h])=0$, for all elements $g,h \in P_J$, Lemma~\ref{l:ni}(i) implies that $N_0(f)=N_1(f)=0$, for every element $f \in [P_J,P_J]$.
\end{proof}

\begin{proof}[Proof of Proposition~\ref{p:no-adad}]
Since $d-1 \in J'$, we have $N_0([a_0,a_{d-1}]) = 1$. Therefore, by Lemma~\ref{l:eveneven}, the element $[a_0,a_{d-1}]$ cannot be in the commutator of $P_J$.
\end{proof}

\subsection{Proof of Theorem~\ref{t:main}}\label{s:main-proof}

The proof of Theorem~\ref{t:main} uses the following sufficient condition.

\begin{theorem}[Bondarenko and Samoilovych~\cite{bondarenko-s:size3}]\label{t:bs-condition}
Let $G_P$ be a finitely constrained group of binary tree automorphisms defined by an essential pattern group $P$ of pattern size $d$, $d \geq 2$. If $[P,P]$ does not contain $P_{d-1}$ then $G_P$ is not topologically finitely generated.
\end{theorem}

\begin{proof}[Proof of Theorem~\ref{t:main}]
By Theorem~\ref{t:main-equiv}, $P$ is a maximal subgroup of $G(d)$ that does not contain $a_{d-1}$ and contains the commutator subgroup of $G(d)$. In particular $P$ contains $[a_0,a_{d-1}]$. On the other hand, $[a_0,a_{d-1}]$ stabilizes level $d-1$, which means that $[a_0,a_{d-1}] \in P_{d-1}$. By Proposition~\ref{p:no-adad}, $[a_0,a_{d-1}] \not\in [P,P]$. Therefore, by Theorem~\ref{t:bs-condition}, $G_P$ is not topologically finitely generated.
\end{proof}


\newcommand{\etalchar}[1]{$^{#1}$}
\def\cprime{$'$}


\end{document}